\documentclass[11pt,bezier]{article}
\usepackage{amsmath,amssymb,amsfonts,euscript,graphicx}

\newtheorem{preproof}{{\bf \indent Proof.}}

\newenvironment{proof}[1]{\begin{preproof}{\rm
               #1}\hfill{$\Box$}}{\end{preproof}}

\newtheorem{cor}{\bf\indent Corollary}[section]
\newtheorem{example}{\bf\indent Example}[section]
\newtheorem{thm}{{\bf\indent Theorem}}[section]

\newtheorem{lem}{\bf\indent Lemma}[section]

\title{\bf \large  Computing the strong metric dimension for \\co-maximal ideal graphs of commutative rings}

\author{{\normalsize  {\sc R. Shariyari${}^{\mathsf{a}}$,   R.  Nikandish${}^{\mathsf{b,c}}$\thanks{Corresponding author}, A. Tehranian${}^{\mathsf{a}}$, H. Rasouli${}^{\mathsf{a}}$}
}\vspace{3mm}\\
{\footnotesize{${}^{\mathsf{a}}$\it Department of Mathematics, Science and Research
Branch,}}\\
{\footnotesize{${}^{\mathsf{}}$\it Islamic Azad University, Tehran, Iran}}\\
{\footnotesize{${}^{\mathsf{b}}$\it  Department of Mathematics,
Jundi-Shapur University of Technology,  Dezful,
Iran}}\\
{\footnotesize{${}^{\mathsf{c}}$\it  Jundi-Shapur Research Institute,
Jundi-Shapur University of Technology,  Dezful,
Iran}}\\
{\footnotesize{${}^{\mathsf{}}$\it}}\\
{\footnotesize{$\mathsf{phd.sh2017@gmail.com}$\quad\quad
$\mathsf{r.nikandish@ipm.ir}$\quad\quad$\mathsf{tehranian@srbiau.ac.ir}$\quad\quad$\mathsf{hrasouli@srbiau.ac.ir}$}}}

\begin{document}

\maketitle

%%%%%%%%%%%%%%%%%%%%%%%%%%%%%%%%%%%%%%%%%%%%%%%%%%%%%%%%%%%%%%%%%%%%%%%%%%%%%%%%%%%%%%%%%%%%%%%%%%%%%%%%%%%%%%%%%%%%%%%%%%%%%%%%%%%%%%%%%%%%%%%%%%%%%%%

%%%%%%%%%%%%%%%%%%%%%%%%%%%%%%%%%%%%%%%%%%%%%%%%%%%%%%%%%%%%%%%\begin{tabbing}
%%%%%%%%%%%%%%%%%%%%%%%%%%%%%%%%%%%%%%%%%%%%%%%%%%%%%%%%%%%%%%%%%%%%%%%%%%%%%%%%%%%%%%%%
\begin{abstract}
{\small Let $R$ be a commutative ring with identity.
The co-maximal ideal graph of $R$, denoted by $\Gamma(R)$, is a simple graph whose vertices are proper ideals of $R$ which are not contained in the Jacobson radical of $R$ and two distinct vertices $I, J$ are adjacent  if and only if  $I+J=R$.
In this paper, we use Gallai$^{^,}$s Theorem and the concept of strong resolving graph to compute the strong metric  dimension for co-maximal ideal graphs of commutative rings. Explicit formulae for  the strong metric dimension,  depending on whether the ring is reduced or not,  are established.}
\end{abstract}
%%%%%%%%%%%%%%%%%%%%%%%%%%%%%%%%%%%%%%%%%%%%%%%%%%%%%%%%%%%%%%%%%%%%%%%%%%%%%%%%%%%%%%%%%%%%%%%%%%%%%%%%%%%%%%%%%%%%%%%%%%%%%%%%%%%%%%%%%%%%%%%%%%%%%%%
\begin{center}\section{Introduction}\end{center}
\par
 The concept of metric dimension which enables an observer to uniquely recognize the current position of  a moving enemy in a network  was initiated by
  Harary and Melter \cite{Harary}. This parameter has found  various applications in the other fields  of sciences (see for instance \cite{Ollermann, [8]}).
  From then on many graph theorists have been attracted
  by computing the metric dimension of graphs (see for example \cite{Bu, Im, Ji, kla}).  In 2004, Seb\"{o} and Tannier (\cite{seb}) introduced a more restricted parameter than metric dimension called strong metric dimension. Computing the strong metric dimension of graphs has appeared in some publications
   (see  \cite{kuz1, kuz2, oller} for more information). Both of the strong background and  wide range of applications
   motivate some of algebraic graph theorists to study metric and strong metric dimensions of graphs associated
   with algebraic structures, see  \cite {Bai, [6], dolz, eb, imran, ma, ma2, nik1, Pirzada1}.  In this paper, we deal with the problem of finding the
   strong metric dimension
   for co-maximal graph associated with a commutative ring.
 \par

Throughout this paper, all rings $R$ are commutative with identity. The sets of all  maximal ideals, ideals with non-zero annihilators,
Jacobson radical and nilpotent elements of $R$
are denoted by $\mathrm{Max}(R)$, $A(R)$, $J(R)$, $Nil(R)$, respectively.
If $T$ is a subset of a ring $R$, then the symbol $T^*$ denotes $T\setminus\{0\}$.
 Moreover, $Ann_R(T)=\{r\in R|\, rT=0\}$. The ring $R$ is called \textit{reduced} if $Nil(R)=\{0\}$.
%For any undefined notation or terminology in ring theory, we refer the reader to \cite{ati}.

By $G=(V,E)$, we mean a simple and undirected graph $G$ with the vertex set $V=V(G)$ and edge set $E=E(G)$.
\textit{A connected graph} is a graph in which there exists at least one path between any two vertices. \textit{Distance} between two distinct vertices $x,y$, denoted by $d(x,y)$, is the length of the shortest path
 between $x$ and $y$  and
 diam$(G)=max\{d(x,y)\,|x,y\in V\}$ is called the $\it diameter$ of $G$. In the graph $G$, let $V_0\subseteq V$ and $E_0 \subseteq E$.
  \textit{The induced subgraph by} $V_0$,  denoted by $G[V_0]$, is a subgraph of $G$ whose vertex set and edge set are $V_0$  and $E_0=\{\{u,v\}\in E\, |\,u,v\in V_0\}$, respectively.
 Let $x\in V$. Then the open and closed neighborhood of $x$ are denote by $N(x)$ and  $N[x]$, respectively. A  $\it complete$ graph is a graph in which each pair of vertices are adjacent. We use  $K_n$ to denote a complete graph of order $n$. A set $S$ of vertices of a graph $G$ is a \textit{vertex cover of} $G$ if every edge of $G$ has one end
in $S$. \textit{The vertex cover
number of} $G$, denoted by $\alpha(G)$, is the smallest cardinality of a vertex cover of $G$.
\textit{The independence number} of a graph $G$, denoted by $\beta(G)$, is the largest cardinality of an independent set. For a graph $G$, $S\subseteq V(G)$ is
called a $\it clique$ if the induced subgraph
  on $S$ is complete. The number of vertices in the largest clique of a graph $G$ is called the $\it clique\,\, number$ of $G$ and denoted by $\omega(G)$. For a connected graph $G$, let $S=\{v_1,v_2,\dots,v_k\}$ be an ordered subset of $V(G)$ and $v\in V(G)\setminus S$.
  \textit{The metric representation of $v$ with respect to} $S$ is the $k$-vector $D(v|S)=(d(v,v_1),d(v,v_2),\dots, d(v,v_k))$. For $S\subseteq V(G)$, if
  for every $u,v\in V(G)\setminus S$, $D(u|S)=D(v|S)$ implies that $u=v$, then $S$ is called  a \textit {resolving set for} $G$.
 \textit {The metric basis for} $G$ is a resolving set $S$ of  minimum cardinality and the number of elements  in $S$ is called  the
 \textit {metric dimension of} $G$ and denoted by $dim_M(G)$.
A vertex $w$ of a connected graph $G$ \textit{strongly resolves} two vertices $u, v$ of $G$ if there
exists a shortest path from $u$ to $w$ containing $v$ or a shortest path from $v$ to $w$  containing $u$. A set $S$ of vertices is a \textit{strong
resolving set for} $G$ if every pair of vertices of $G$ is strongly resolved by some vertex of $S$. The smallest cardinality
of a strong resolving set for $G$ is called the \textit{strong metric dimension of} $G$ and denoted by $sdim_M(G)$.
For all undefined notions from graph theory, we refer the reader to \cite{west}.

Let $R$ be a ring. \textit {The co-maximal ideal graph of} $R$, denoted by $\Gamma(R)$, is a graph whose
 vertices are proper ideals of $R$ which are not contained in the Jacobson radical of $R$ and two distinct vertices $I, J$ are adjacent  if and only if  $I+J=R$. The concept of co-maximal ideal graph of a commutative ring was first introduced and studied in  \cite{ye}.
 Since then co-maximal ideal graphs of commutative rings have been studied by several authors, for instance see  \cite {ak, wu, ye2}. It is worth mentioning  that
 co-maximal graphs for lattices in \cite{af}, for groups in \cite{ak1}, for matrix algebras in \cite{mira}, for two generated groups in \cite{mir}   and for non-commutative rings in \cite{shen}
 were investigated. This paper is devoted to study the strong metric dimension of a co-maximal ideal graph and it is organized as follows.
 In Section 2, we completely determine $sdim_M(\Gamma(R))$
  in terms of the number of maximal ideals of $R$, in case $R$ is reduced. In Section 3, we focus on  the strong metric dimension of   $\Gamma(R)$,
  when $R$ is a non-reduced ring.

{\begin{center}{\section{Reduced rings case}}\end{center}}\vspace{-2mm}
%%%%%%%%%%%%%%%%%%%%%%%%%%%%%%%%%%%%%%%%%%%%%%%%%%%%%%%%%%%%%%%%%%%%%%%%%%%%%%%%%%%%%%%%%%%%%%%%%%%%%%%%%%%%%%%%%%%%%%%%%%%%%%%%%%%%%%%%%%%%%%%%%%%%
In this section, we present   strong metric dimension formula for a co-maximal ideal graph, when $R$  is reduced. We begin with a series of lemmas.
\begin{lem}\label{dimhelpnonred}
Let $G$ be a connected graph and $diam(G)=d<\infty$. Then the following statements hold.

$(1)$ $dim_M(G)$ is finite if and only if  $G$ is finite.

$(2)$ If  $W\subset V(G)$ is a strong resolving set of $G$ and $u,v\in V(G)$ such that $N(u)=N(v)$ or $N[u]=N[v]$, then
$u\in W$ or $v\in W$.

$(3)$ If  $W\subset V(G)$ is a strong resolving set of $G$ and $u,v\in V(G)$ such that $d(u,v)=diam(G)$, then
$u\in W$ or $v\in W$.
\end{lem}
\begin{proof}
{$(1)$ To prove the non-trivial direction,
 assume that $dim_M(G)$ is finite and for some non-negative integer $n$, let $W=\{I_1,I_2,\dots,I_n\}$ be a metric basis for $G$. Since  $diam(G)=d< \infty$, for every $x\in V(G)\setminus W$, there are only $(d+1)^n$ choices  for $D(x|W)$. Thus $|V(G)|\leq (d+1)^n+n$ and hence  $G$ is finite.

$(2)$ and $(3)$ are obvious.
}
\end{proof}

Let $G$ be a graph. It is easily seen that every strong resolving set is also a  resolving set, which leads to $dim_M(G)\leq sdim_M(G)$.
Hence, we have the following immediate corollary.
  \begin{cor}\label{dimhelpnonreds}
Let $R$ be a ring. Then  $sdim_M(\Gamma(R))$ is finite if and only if  $\Gamma(R)$ is finite.
\end{cor}

The following well-known result, due to Gallai, which states a relationship between the  independence number and the vertex cover number of a graph $G$ has a key role in this paper.

\begin{lem}\label{Gallai}
{\rm (Gallai$^{^,}$s Theorem)} For any graph $G$ of order $n$, $\alpha(G)+\beta(G)=n$.
\end{lem}

A vertex $u$ of $G$ is maximally distant from $v$ (in $G$) if for every $w\in N(u)$, $d(v,w)\leq d(u,v)$. If $u$ is maximally distant from $v$ and $v$ is maximally distant from $u$, then we say that $u$ and $v$ are  mutually  maximally distant. The boundary of $G$ is defined as

$\partial(G)=\{u\in V(G)|\,\, {\rm there \,\,exists}\, v\in V(G) \, {\rm such \,\,that} \,\,u,v\,\, {\rm are \, mutually \,  maximally\, distant} \}$.

We use the notion of strong resolving graph introduced by Oellermann and Peters-Fransen in \cite{oller}.
The strong resolving graph of $G$ is a graph $G_{SR}$ with vertex set $V(G_{SR})=\partial(G)$ where two vertices $u,v$ are adjacent in $G_{SR}$  if and only if $u$ and $v$ are mutually  maximally distant.

It was shown in \cite{oller} that the problem of finding the strong metric dimension of a graph $G$ can be transformed into the problem of computing the vertex cover number of $G_{SR}$.

\begin{lem}\label{Oellermann}
{\rm (\cite[Theorem 2.1]{oller})} For any connected graph $G$, $sdim_M(G)=\alpha(G_{SR})$.
\end{lem}

 The next example illustrates the validity of Lemma \ref{Oellermann}.

\begin{example}
\end{example}

$(1)$  Since $(K_n)_{SR}=K_n$,  $sdim_M(K_n)=n-1$.

$(2)$ Let $R=\mathbb{Z}_2\times \mathbb{Z}_2\times \mathbb{Z}_2$. Suppose that
$X=\{(\mathbb{Z}_2,\mathbb{Z}_2,0),(0,\mathbb{Z}_2,\mathbb{Z}_2),(\mathbb{Z}_2,0,\mathbb{Z}_2)\}$ and $Y=\{(\mathbb{Z}_2,0,0),(0,\mathbb{Z}_2,0),(0,0,\mathbb{Z}_2)\}$.
It is not hard to see that
for any $u\in X$, there is no $v\in V(\Gamma(R))$ such that $u$ and $v$ are  mutually  maximally distant, whereas
each pair of vertices in $Y$ are mutually  maximally distant.  This follows that
$\partial(\Gamma(R))=\{(\mathbb{Z}_2,0,0),(0,\mathbb{Z}_2,0),(0,0,\mathbb{Z}_2)\}$ and $\Gamma(R)_{SR}=K_3$. Since
$\alpha(\Gamma(R)_{SR})=2$, Lemma \ref{Oellermann} follows that  $sdim_M(\Gamma(R))=2$.
 On the other hand,
$W=\{(\mathbb{Z}_2,0,0),(0,\mathbb{Z}_2,0) \}$
is the  minimum strong resolving set, i.e., $sdim_M(\Gamma(R))=2$.

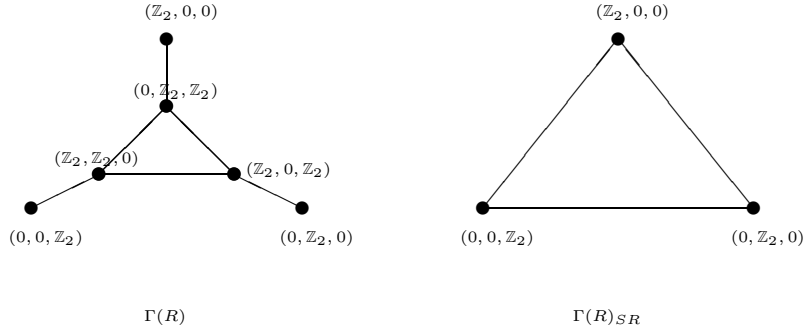
\unitlength=1.5mm
\begin{figure}[!th]
\centering
\begin{picture}(60,40)(20,-20)
\put (28,12){\circle*{1.2}}
\put (26,14){\tiny{$(\mathbb{Z}_2,0,0)$}}
\put (28,6){\line (0,1){6}}
\put (28,6){\circle*{1.2}}
\put (25,7){\tiny{$(0,\mathbb{Z}_2,\mathbb{Z}_2)$}}
\put (22,0){\circle*{1.2}}
\put (18,1){\tiny{$(\mathbb{Z}_2,\mathbb{Z}_2,0)$}}
\put (22,0){\line (1,1){6}}
\put (22,0){\line (1,0){12}}
\put (22,0){\line (1,0){12}}
\put (34,0){\circle*{1.2}}
\put (35,0){\tiny{$(\mathbb{Z}_2,0,\mathbb{Z}_2)$}}
\put (34,0){\line (-1,1){6}}
%\put (34,0){\line (-1,2){6}}
\put (16,-3){\circle*{1.2}}
\put (14,-6){\tiny{$(0,0,\mathbb{Z}_2)$}}
\put (16,-3){\line (2,1){6}}
%\put (16,-3){\line (4,3){12}}
\put (40,-3){\circle*{1.2}}
\put (38,-6){\tiny{$(0,\mathbb{Z}_2,0)$}}
\put (40,-3){\line (-2,1){6}}
%\put (40,-3){\line (-6,1){18}}
%\put (20,-13){\tiny{\rm ${G_{SR} $}}\\

\put (68,12){\circle*{1.2}}
\put (66,14){\tiny{$(\mathbb{Z}_2,0,0)$}}
\put (28,6){\line (0,1){6}}
\put (28,6){\circle*{1.2}}
%\put (25,7){\tiny{$V_2$}}
\put (22,0){\circle*{1.2}}
%\put (18,1){\tiny{$V_4$}}
\put (22,0){\line (1,1){6}}
\put (22,0){\line (1,0){12}}
\put (22,0){\line (1,0){12}}
\put (34,0){\circle*{1.2}}
%\put (35,0){\tiny{$V_3$}}
\put (34,0){\line (-1,1){6}}
%\put (34,0){\line (-1,2){6}}
\put (56,-3){\circle*{1.2}}
\put (54,-6){\tiny{$(0,0,\mathbb{Z}_2)$}}
\put (56,-3){\line (1,0){24}}
\put (56,-3){\line (4,5){12}}
%\put (16,-3){\line (4,3){12}}
\put (80,-3){\circle*{1.2}}
\put (78,-6){\tiny{$(0,\mathbb{Z}_2,0)$}}
\put (80,-3){\line (-4,5){12}}
%\put (40,-3){\line (-6,1){18}}
\put (26,-13){\tiny{\rm $\Gamma(R) $}}
\put (64,-13){\tiny{\rm $\Gamma(R)_{SR} $}}
\end{picture}
% \caption{\rm $\Gamma(R)$ and $\Gamma(R)_{SR}$} \label{figure:fr}
\end{figure}

It was proved in \cite[Theorem 2.4]{ye} that diam($\Gamma(R))\leq 3$. Hence we omit the elementary proof of the next lemma.

\begin{lem}\label{lemma2d}

 Let $R\cong F_1\times \cdots\times F_n$, where  $F_i$ is a field for every $1\leq i\leq n$
 and let $I$,  $J$ be two vertices of $\Gamma(R)$. Then the following statements hold.

$(1)$   $d(I,J)_{\Gamma(R)}=2$ if and only if $I\cap J\neq 0$ and $I+J\neq R$.

 $(2)$ $d(I,J)_{\Gamma(R)}=3$ if and only if $I\cap J= 0$ and $I+J\neq R$.

 $(3)$ If $I\subseteq J$ or $J\subseteq I$, or $I+J=R$, then  $I$ and $J$ are not mutually  maximally distant.
\end{lem}
In order to present our results we need to introduce some more terminologies.

Let $R$ be a ring and $\Gamma(R)$ be the co-maximal ideal graph of $R$.
We define $\Gamma(R)^{**}$ as the graph with vertex set $V(\Gamma(R)^{**})=V(\Gamma(R))$ such
that two vertices $I,J$ are adjacent in $\Gamma(R)^{**}$ if and only if $I+J\neq R$, $I\nsubseteq J$ and $J\nsubseteq I$.
Also, we define $\Gamma(R)^{*}$ as follows: Let $\Gamma(R)^*=\Gamma(R)$, if  $\Gamma(R)$ is complete, and otherwise $\Gamma(R)^*$ is obtained from $\Gamma(R)^{**}$
by removing all its isolated vertices.

Let  $R\cong R_1\times\cdots \times R_n$, where  $R_i$ is a ring for every $1\leq i\leq n$. Let $I=(I_1,I_2,\dots,I_n)$
be an ideal of $R$. By $NZC(I)$, we mean the number of nonzero components of $I$.

\begin{lem}\label{e32e}
Let $R$ be a reduced ring and $sdim_M(\Gamma(R))<\infty$.  Then the following statements hold.

$(1)$ If $|\mathrm{Max}(R)|=2$, then   $\Gamma(R)^{**}=\Gamma(R)^*=\Gamma(R)=\Gamma(R)_{SR}=K_2$.

$(2)$
If $|\mathrm{Max}(R)|\geq 3$, then $I$ is an isolated vertex in $\Gamma(R)^{**}$ if and only if   $I\in \mathrm{Max}(R)$.

 $(3)$ If $|\mathrm{Max}(R)|=n\geq 3$, then  $\Gamma(R)^{**}=H+K_{1}+K_{1}+\cdots+K_{1}$ ($n$ times), where $H$ is a connected graph.

$(4)$ $\Gamma(R)^*=\Gamma(R)_{SR}$.
\end{lem}

\begin{proof}
{$(1)$ Since $sdim_M({\Gamma}(R))$ is finite, $R$ has  finitely many ideals, by Corollary \ref{dimhelpnonreds} and so
 \cite[Theorem 8.7]{ati} implies that $R\cong F_1\times\cdots \times F_n$, where  $F_i$ is a field for every $1\leq i\leq n=|\mathrm{Max}(R)|$.
 If $n=2$,
since  $\Gamma(R)=K_2$, we have  $\Gamma(R)^{**}=\Gamma(R)^*=\Gamma(R)=\Gamma(R)_{SR}=K_2$.

$(2)$ First we show that $I$ is an isolated vertex in $\Gamma(R)^{**}$, for every  $I\in \mathrm{Max}(R)$.  Assume that
$I$ is  adjacent to $J$ in $\Gamma(R)^{**}$. Then $I\nsubseteq J$ and $J\nsubseteq I$. Hence $I+J= R$, a contradiction.
Let $A= V(\Gamma(R)^{**})\setminus  \mathrm{Max}(R)$. To complete the proof, we show that $\Gamma(R)^{**}[A]$ is a connected graph. To see this,
let   $S=\{(F_1,0,\dots,0),(0,F_2,0,\dots,0),\ldots,(0,\dots,0,F_n)\}$. If $I,J\in S$, then   $I+J\neq R$, $I\nsubseteq J$ and $J\nsubseteq I$. Thus $I$ is  adjacent to $J$ in $\Gamma(R)^{**}[A]$.
Therefore, the induced subgraph on $S$ is a clique in $\Gamma(R)^{**}[A]$.
Hence suppose that  $I \in V(\Gamma(R)^{**}) \setminus S\cup\mathrm{Max}(R)$. Then    $I$ is adjacent to some vertices of $S$ in
$\Gamma(R)^{**}[A]$, because at least two  components of $I$ are zero.

 $(3)$ By $(2)$, it is obvious.

$(4)$ First we claim that $ V(\Gamma(R)^*)=V(\Gamma(R)_{SR})$. By parts (2) and (3),  $ V(\Gamma(R)^*)=V({\Gamma}(R))\setminus \mathrm{Max}(R)$. We prove that $V(\Gamma(R)_{SR})= V({\Gamma}(R))\setminus \mathrm{Max}(R)$. Let $I\in \mathrm{Max}(R)$. It is shown that  there is no $J\in V(\Gamma(R))$ such that $I$ and $J$ are  mutually  maximally distant. If not, $I\nsubseteq J$ and $J\nsubseteq I$, by Lemma \ref{lemma2d}. Since $I\in \mathrm{Max}(R)$, $I+J=R$ in $\Gamma(R)$. Hence $d(I,J)=1$. If $J\neq Ann(I)$, then $d(I,Ann(I))=1$, but $Ann(I)\subseteq J$ and so
$d(J,Ann(I))\neq1$, a contradiction. If $J= Ann(I)$, for some $Ann(I)\subseteq K$, then $d(I,Ann(I))=1$, but $Ann(I)\subseteq K$ and so
$d(K,Ann(I))\neq1$, a contradiction. Hence  $I$ is an isolated vertex in $\Gamma(R)_{SR}$ and so $V(\Gamma(R)_{SR})\cap\mathrm{Max}(R)=\varnothing$. Next, assume that $S=\{I\in V({\Gamma}(R))\,|\, NZC(I)=1\}$ and
 $I,J\in S$. Then
$d(I,J)_{{\Gamma}(R)}=3=diam({\Gamma}(R))$. This implies that $I,J$ are mutually  maximally distant and hence $I$ is  adjacent to $J$ in ${\Gamma}(R)_{SR}$.
Therefore, the induced subgraph on $S$ is a clique in ${\Gamma}(R)_{SR}$.
  Now, we show that if $I\not\in \mathrm{Max}(R)\cup S$, then $I$ is adjacent to some  vertices of $S$  in
${\Gamma}(R)_{SR}$.  But this is obvious, because it is not hard to see that $d(I,J)=3=daim(\Gamma(R))$, for some $J\in S$.
Therefore,
$\partial({\Gamma}(R))=V({\Gamma}(R))\setminus \mathrm{Max}(R)$ and so the claim is proved.
To complete the proof, it is enough to show that the adjacency between vertices in $\Gamma(R)^*$ is in a one to one correspondence between vertices in  $\Gamma(R)_{SR}$ and vice versa.
 Assume that $I,J\in V(\Gamma(R)^*)$ and $I$ is adjacent to $J$. Thus $I+J\neq R$, $I\nsubseteq J$ and $J\nsubseteq I$.
Since $I+J\neq R$, $d(I,J)_{\Gamma(R)}\neq 1$. Indeed,  $d(I,J)_{\Gamma(R)}\in \{2,3\}$. If $d(I,J)_{\Gamma(R)}=3$, then $d(I,J)=daim(\Gamma(R))$. So
$I,J$ are mutually  maximally distant. Hence $I$ is adjacent to $J$ in
$\Gamma(R)_{SR} $. Therefore, suppose that $d(I,J)_{\Gamma(R)}=2$ and  $K\in N_{\Gamma(R)}(I)$.
Since $I+K=R$ and $J\nsubseteq I$, $K\cap J\neq 0$. By Lemma \ref{lemma2d},  $d(K,J)_{\Gamma(R)}\leq 2$. Thus
$d(K,J)_{\Gamma(R)}\leq d(I,J)_{\Gamma(R)}$.
 Similarly, $d(L,I)_{\Gamma(R)}\leq d(I,J)_{\Gamma(R)}=2$, for every $L\in N(J)$. So
$I,J$ are mutually  maximally distant and hence
 $I$ is adjacent to $J$ in
$\Gamma(R)_{SR} $.
Finally,l let $I,J\in V(\Gamma(R)_{SR})$ and $I$ is adjacent to $J$.  Thus $I,J$ are mutually  maximally distant.
By Lemma \ref{lemma2d}, $I+J\neq R$, $I\nsubseteq J$ and $J\nsubseteq I$ and so
 $I$ is adjacent to $J$ in $\Gamma(R)^*$.
}
\end{proof}

\begin{lem}\label{e32}
Suppose that $R\cong F_1\times\cdots \times F_n$, where  $F_i$ is a field for every $1\leq i\leq n$ and $S$  is the largest independent
set of $\Gamma(R)_{SR}$. Then the followings hold.

$(1)$  $\partial({\Gamma}(R))=V({\Gamma}(R))\setminus \mathrm{Max}(R)$.

$(2)$ $NZC(I)=n-2$, for some $I\in S$.

$(3)$  $\beta({\Gamma}(R)_{SR})=n-2$, if $n\geq 3$.
\end{lem}
\begin{proof}
{$(1)$ By Lemma \ref{e32e}, it is obvious.

$(2)$
 Assume that $S=\{I_1,I_2,\dots,I_t\}$
is the largest independent
set of $\Gamma(R)_{SR}$. Clearly,
$1\leq NZC(I)\leq n-2$, for every $I\in V({\Gamma}(R)_{SR})$. With no loss of generality, assume that $NZC(I_t)\geq NZC(I_i)$ for every $1\leq i\leq t$ and $I_i\in S$. We claim that $ NZC(I_t)= n-2$.
Assume to the contrary, $ NZC(I_t) \leq n-3$.
Since $I_t$ is not  adjacent to $I_i$, for every $1\leq i\leq t$ and $I_i\in S$, we deduce that $I_i\subseteq I_t$ or $I_t\subseteq I_i$ or $I_t+I_i=R$.
If  $I_t\subseteq I_i$, for some $I_i\in S$, then  $NZC(I_i)\geq NZC(I_t)$, a contradiction. This implies that
 $I_i\subseteq I_t$  or $I_t+I_i=R$, for every  $I_i\in S$ and $1\leq i\leq t$.
 Assume that $S_1=\{I\in S\,\,|\,\,I\subseteq I_t\}$ and $S_2=\{I\in S\,\,|\,\,I+ I_t=R\}$.
Since $ NZC(I_t) \leq n-3$, by replacing one of the zero components of $I_t$ by $F_i$, we get $J\in V({\Gamma}(R)_{SR})$.
Hence $I\subseteq J$, for every $I\in S_1$ and $I+J=R$, for every $I\in S_2$. This implies that $S\cup \{J\}$
is a independent set of $\Gamma(R)_{SR}$, a contradiction. Therefore,  $ NZC(I_t)= n-2$ which completes the proof.

$(3)$
 Assume that

$A_1=\{I\in V({\Gamma}(R)_{SR})\,|\, NZC(I)=1\}$,

$A_2=\{I\in V({\Gamma}(R)_{SR})\,|\, NZC(I)=2\}$,

$A_3=\{I\in V({\Gamma}(R)_{SR})\,|\, NZC(I)=3\}$,

$\vdots$

$A_{n-2}=\{I\in V({\Gamma}(R)_{SR})\,|\, NZC(I)=n-2\}$. Consider the following facts:

{\bf{Fact 1.}} For every $1 \leq i\leq n-2$ and $I,J\in A_i$, since $NZC(I)=NZC(J)$, we conclude that $I\nsubseteq J$ and $J\nsubseteq I$.

{\bf{Fact 2.}} If $I,J\in A_i$ for every $1 \leq i\leq n-2$ and
$I$ is not  adjacent to $J$, then $I+J=R$, by Fact 1.

{\bf{Fact 3.}} Since there is no $I,J\in A_i$ such that $I+J=R$, for every $1 \leq i\leq [\dfrac{n-2}{2}]-1$, we conclude that
${\Gamma}(R)_{SR}[A_i]$ is complete, by Fact 2.

{\bf{Fact 4.}} For every $[\dfrac{n-2}{2}] \leq i\leq n-2$, assume that $S_i\subseteq A_i$ is the largest set of $A_i$ such that for every $I,J\in S_i$, $I+J=R$. Then $|S_i|=[\dfrac{n}{n-i}]$.

Continue the proof in the following steps:

{\bf{Step 1.}}
Put $I_1=(F_1,0,\dots,0), I_2=(F_1,F_2,0,\dots,0), \dots, I_i=(F_1,F_2,\dots,F_i,\dots,0)$, where $i=[\dfrac{n-2}{2}]-1$ and
$W_1=\{I_1,I_2,\dots,I_i\}$. Then $W_1$ is an independent
set of ${\Gamma}(R)_{SR}[\cup A_i]$. Since by Fact 3 ${\Gamma}(R)_{SR}[A_i]$ is complete and $|W_1\cap A_i|=1$, we deduce that
$W_1$
is the largest independent
set of ${\Gamma}(R)_{SR}[\cup A_i]$, for $i=1$ up to $i=[\dfrac{n-2}{2}]-1$ (if $n$ is odd up to $[\dfrac{n-2}{2}])$. This implies that $\beta({\Gamma}(R)_{SR}[\cup A_i])=|W_1|=[\dfrac{n-2}{2}]-1$.

{\bf{Step 2.}} Consider $A_i$, where $i=[\dfrac{n-2}{2}]$ and $n$ is even. Then we have  $|S_i|=2$.
Without loss of generality, assume that $S_i=\{I=(F_1,F_2,\dots,F_i,\dots,0),J=(0,\dots,0,F_{i+1},\dots,F_n) \}$.
This implies that $J$ is  adjacent to some  vertices of $W_1$, but $W_2=W_1\cup \{I\}$
is the largest independent
set of ${\Gamma}(R)_{SR}[\cup A_i]$ for $i=1$ up to $i=[\dfrac{n-2}{2}]$. Hence
 $\beta({\Gamma}(R)_{SR}[\cup A_i])=|W_2|=[\dfrac{n-2}{2}]$.

{\bf{Step 3.}} Continue the procedure in Step 2 up to $t=n-2-[\dfrac{n-2}{2}]$ and get
$W_t=\{(F_1,0,\dots,0),(F_1,F_2,0,\dots,0),\dots,(F_1,F_2,\dots,F_{n-2},0,0) \}$. Therefore,
 $$\beta({\Gamma}(R)_{SR}[\cup_{1}^{n-2} A_i])=\beta({\Gamma}(R)_{SR})=|W_t|=n-2.$$
}
\end{proof}
We are now in a position to state the main result of this section.

\begin{thm}\label{dimprod}
Suppose that $R$ is a reduced ring. If  $sdim_M({\Gamma}(R))$ is finite, then

$(1)$ If  $|\mathrm{Max}(R)|=2$, then $sdim_M({\Gamma}(R))=1$.

$(2)$ If  $|\mathrm{Max}(R)|=n\geq 3$, then
 $sdim_M({\Gamma}(R))=2^n-2n$.
\end{thm}
\begin{proof}
{Since $sdim_M({\Gamma}(R))$ is finite, $R$ has  finitely many ideals, by Corollary \ref{dimhelpnonreds}, and so \cite[Theorem 8.7]{ati} implies that
$R\cong F_1\times\cdots \times F_n$, where  $F_i$ is a field for every $1\leq i\leq n=|\mathrm{Max}(R)|$.
If $n=2$, then $\Gamma(R)=K_2$ and so $sdim_M({\Gamma}(R))=1$. If $n\geq 3$, then
by Lemmas \ref{Gallai} and \ref{Oellermann}, $$sdim_M({\Gamma}(R))=\alpha({\Gamma}(R)_{SR})=V({\Gamma}(R)_{SR})-\beta({\Gamma}(R)_{SR}).$$
On the other hand, $\partial({\Gamma}(R))=2^{n}-2-n$ and $\beta({\Gamma}(R)_{SR})=n-2$, by Lemma \ref{e32}. Therefore,
 $sdim_M({\Gamma}(R))=2^{n}-2-n-(n-2)=2^n-2n$.
}
\end{proof}
{\begin{center}{\section{Non-reduced rings case}}\end{center}}\vspace{-2mm}
%%%%%%%%%%%%%%%%%%%%%%%%%%%%%%%%%%%%%%%%%%%%%%%%%%%%%%%%%%%%%%

In this section, we study the strong metric dimension of  $\Gamma(R)$, when $R$ is non-reduced.

As we have seen in Corollary \ref{dimhelpnonreds}, $sdim_M(\Gamma(R))$ is finite if and only if  $\Gamma(R)$ is finite. Hence in this section, we focus on rings with finitely many ideals. Obviously, such rings are Artinian and it follows from \cite[Theorem 8.7]{ati} that there exists a positive integer $n$ such that $R\cong R_1\times \cdots\times R_n$, where $R_i$ is an Artinian local ring for every $1\leq i\leq n$. If every $R_i$ is non-reduced, then it is not a field and so $|A(R_i)^*|\geq 1$ (It is not hard to see that if $R$ is Artinian and non-reduced, then every proper ideal has a non-zero annihilator). Now, let $R$ be a such ring. Suppose that  $I= (I_1,\dots, I_n)$ and $J= (J_1,\dots,J_n)$ are two vertices of $\Gamma(R)$. Define the relation  $\thicksim$ on  $A(R)^*$ as follows: $I\thicksim J$, whenever for each $1\leq i\leq n$,
 ``$I_i\subseteq Nil(R_i)$ if and only if $J_i\subseteq Nil(R_i)$''.
   Clearly, $\thicksim$ is an equivalence relation on $A(R)^*$.
   The equivalence class of $I$ is denoted by
   $[I]$.   Suppose that  $I= (I_1,\dots, I_n)$ is and ideal of $R$. Then  by $I^{\prime}= (I_1^{\prime},\dots, I_n^{\prime})$, we mean a new ideal obtained from $I$ whose  all nonzero nilpotent components are replaced  by $0$.

   We define $\Gamma(R)^{\prime}$ as the graph with vertex set $V(\Gamma(R)^{\prime})=V(\Gamma(R))$ such that two vertices $I,J$ are adjacent in $\Gamma(R)^{\prime}$ if and only if $[I]=[J]$ or $I^{\prime}+J^{\prime}\neq R$, $I^{\prime}\nsubseteq J^{\prime}$ and $J^{\prime}\nsubseteq I^{\prime}$.

The proof of the following lemma is obvious.

\begin{lem}\label{e32ezs}
 Suppose that $R\cong R_1\times\cdots \times R_n$, where $R_i$ is an Artinian local ring for every $1\leq i\leq n$.  If $I$ and $J$ are vertices of $\Gamma(R)$ with $[I]=[J]$, then $N[I]=N[J]$.
\end{lem}

\begin{lem}\label{e32ez}
 Suppose that $R\cong R_1\times\cdots \times R_n$, where $R_i$ is an Artinian local ring and $|A(R_i)^*|\geq 1$,  for every $1\leq i\leq n$. Then the following statements hold.

$(1)$   $\Gamma(R)^{\prime}=H^{\prime}+K_{|A(R_1)|}+K_{|A(R_2)|}+\cdots+K_{|A(R_n)|}$, where $H^{\prime}=\varnothing$ if $n=2$ and $H^{\prime}$ is connected if $n\geq 3$.

$(2)$ $\Gamma(R)^{\prime}=\Gamma(R)_{SR}$.
\end{lem}
\begin{proof}
{Let

$A_1=\{I\in V({\Gamma}(R))\,|\, [I]=[(Nil(R_1), R_2,\dots,R_n)]\}$,

$A_2=\{I\in V({\Gamma}(R))\,|\, [I]=[(R_1, Nil(R_2),R_3,\dots,R_n)]\}$,

$\vdots$

$A_n=\{I\in V({\Gamma}(R))\,|\, [I]=[(R_1,\dots,R_{n-1},Nil(R_n))]\}$ and

$A=\cup A_i$.

Obviously, if $I,J\in A_i$, then $[I]=[J]$ and so $I$ is  adjacent to $J$ in $\Gamma(R)^{\prime}$. This implies that
 $\Gamma(R)^{\prime}[A_i]$ is a complete graph, for every $1\leq i\leq n$.
Suppose that $I\in A_i$ and $J\not\in A_i$. We show that $I$ is not  adjacent to $J$ in $\Gamma(R)^{\prime}$.
If not, since  $[I]\neq [J]$, we must have $I^{\prime}+J^{\prime}\neq R$, $I^{\prime}\nsubseteq J^{\prime}$ and $J^{\prime}\nsubseteq I^{\prime}$, a contradiction, as $I^{\prime}+J^{\prime}\neq R$.

Next, it is proved that $\Gamma(R)^{\prime}\setminus  A$ is a connected graph. For this,
let $$S=\{(R_1,0,\dots,0),(0,R_2,0,\dots,0),\ldots,(0,\dots,0,R_n)\}.$$
Using a proof technique similar to  Lemma \ref{e32e} implies that
  the induced subgraph on $S$ is a clique in $\Gamma(R)^{\prime}$
  and every $I \in V(\Gamma(R)^{\prime}) \setminus (S\cup A) $ is adjacent to some vertices of $S$ in
$\Gamma(R)^{\prime}$. Therefore,
$\Gamma(R)^{\prime}=H^{\prime}+K_{|A(R_1)|}+K_{|A(R_2)|}+\cdots+K_{|A(R_n)|}$.
By an easy verification, if $n=2$, then $H^{\prime}=\varnothing$.

$(2)$
 Let $I\in V(\Gamma(R)^{\prime})$. Since $|[I]|\geq 2$,  $I\sim J$ and $I\neq J$. Since $N(I)=N(J)$, we deduce that $I$ and $J$ are  mutually  maximally distant and hence $I\in V(\Gamma(R)_{SR})$.
Thus $ V(\Gamma(R)^{\prime})=V(\Gamma(R)_{SR})$.

We show that $\Gamma(R)^{\prime}=\Gamma(R)_{SR}$.
 Let $I,J\in V(\Gamma(R)^{\prime})$ and $I$ is adjacent to $J$.
If $[I]=[J]$, then  $I,J$ are mutually  maximally distant and so
 $I$ is adjacent to $J$ in
$\Gamma(R)_{SR} $.
If $[I]\neq [J]$, then
$I^{\prime}+J^{\prime}\neq R$, $I^{\prime}\nsubseteq J^{\prime}$ and $J^{\prime}\nsubseteq I^{\prime}$.
By a similar proof to that of Lemma \ref{e32e},
$ I^{\prime},J^{\prime}$ are mutually  maximally distant. Since $[I]=[I^{\prime}]$ and $[J]=[J^{\prime}]$, $I$ is adjacent to $J$ in
$\Gamma(R)_{SR} $.

Finally, suppose that $I,J\in V(\Gamma(R)_{SR})$ and $I$ is adjacent to $J$.
If $[I]=[J]$, then  $N(I)=N(J)$ and
so $I,J$ are mutually  maximally distant.
If $[I]\neq [J]$, then $I+J\neq R$, $I\nsubseteq J$ and $J\nsubseteq I$. Thus
$I^{\prime}+J^{\prime}\neq R$, $I^{\prime}\nsubseteq J^{\prime}$ and $J^{\prime}\nsubseteq I^{\prime}$.
Hence $I$ is adjacent to $J$ in $V(\Gamma(R)^{\prime})$.
}
\end{proof}

%%%%%%%%%%%%%%%%%%%%%%%%%%%%%%%%%%%%%%%%%%%%%%%%%%%%%%%%%%%%%%%%%%%%%%%
\begin{lem}\label{e32z}
 Suppose that $R\cong R_1\times\cdots \times R_n$, where $R_i$ is an Artinian local ring and $|A(R_i)^*|\geq 1$,  for every $1\leq i\leq n$. Then the following statements hold.

$(1)$  $\partial({\Gamma}(R))=V({\Gamma}(R))$.

$(2)$  $\beta({\Gamma}(R)_{SR})=2n-2$.
\end{lem}
\begin{proof}
{$(1)$ is obvious.

$(2)$
By the proof of Lemma \ref{e32ez},
 $\Gamma(R)_{SR}=H^{\prime}+K_{|A(R_1)|}+K_{|A(R_2)|}+\cdots+K_{|A(R_n)|}$.
It is known that $\beta(K_{|A(R_1)|}+K_{|A(R_2)|}+\cdots+K_{|A(R_n)|})=n$.
Let $$A=\{(I_1,\dots,I_n)\in V(H^{\prime})\,\,|\,\, I_i\in \{0,R_1,\dots,R_n\}\,\,\mathrm{for\,\, every}\,\,1\leq i\leq n\}.$$
Then $H^{\prime}[A]\cong H$, where $H=\Gamma(R)^{**}-n(K_{1})$ (as in part (3) of Lemma \ref{e32e}). Thus
 $|A\cap[I]|=1$,  for every equivalence class $[I]$.
On the other hand, since $H^{\prime}[[I]]$ is complete for every $I\in H^{\prime}$,
 $\beta(H^{\prime})=\beta(\Gamma(R)_{SR}[A])$. Since $\Gamma(R)_{SR}[A]\cong H$ and
$\beta(H)=n-2$, we deduce that  $\beta(H^{\prime})=\beta(\Gamma(R)_{SR}[A])=n-2$.
Therefore, $\beta({\Gamma}(R)_{SR})=\beta(H^{\prime}) +n=2n-2$.
}
\end{proof}

Now, we are ready to state the following result.

\begin{thm}\label{isomorphismq}
 Suppose that $R\cong R_1\times\cdots \times R_n$, where $R_i$ is an Artinian local ring and $|A(R_i)^*|\geq 1$,   for every $1\leq i\leq n$. Then
   $sdim_M(\Gamma(R))=|V(\Gamma(R))|-2n+2$.
\end{thm}
\begin{proof}
{ The proof follows from Lemmas \ref{Gallai}, \ref{Oellermann} and \ref{e32z}.}
\end{proof}

The following example is related to Theorem \ref{isomorphismq}.

\begin{example}

\end{example}
   Let $R=\mathbb{Z}_4\times  \mathbb{Z}_4\times  \mathbb{Z}_8$.  Then

$[(\mathbb{Z}_4,\mathbb{Z}_4,(0))]=\{(\mathbb{Z}_4,\mathbb{Z}_4,(0)), (\mathbb{Z}_4,\mathbb{Z}_4,(2)),(\mathbb{Z}_4,\mathbb{Z}_4,(4))\}$,

$[(\mathbb{Z}_4,(0),\mathbb{Z}_8)]=\{(\mathbb{Z}_4,(0),\mathbb{Z}_8),(\mathbb{Z}_4,(2),\mathbb{Z}_8)\}$,

$[((0),\mathbb{Z}_4,\mathbb{Z}_8)]=\{((0),\mathbb{Z}_4,\mathbb{Z}_8),((2),\mathbb{Z}_4,\mathbb{Z}_8)\}$,

$[((0),(0),\mathbb{Z}_8)]=\{((0),(0),\mathbb{Z}_8),((2),(0),\mathbb{Z}_8), ((0),(2),\mathbb{Z}_8),((2),(2),\mathbb{Z}_8)\}$,

$[(\mathbb{Z}_4,(0),(0))]=\{(\mathbb{Z}_4,(0),(0)),(\mathbb{Z}_4,(0),(2)),(\mathbb{Z}_4,(0),(4)),(\mathbb{Z}_4,(2),(0)),
(\mathbb{Z}_4,(2),(2)),
(\mathbb{Z}_4,(2),(4))\}$,

$[((0),\mathbb{Z}_4,(0))]=\{((0),\mathbb{Z}_4,(0)),((0),\mathbb{Z}_4,(2)),((0),\mathbb{Z}_4,(4)),((2),\mathbb{Z}_4,(0)),
((2),\mathbb{Z}_4,(2)),
((2),\mathbb{Z}_4,(4))\}$ and

$A=\{(\mathbb{Z}_4,\mathbb{Z}_4,(0)), (\mathbb{Z}_4,(0),\mathbb{Z}_8), ((0),\mathbb{Z}_4,\mathbb{Z}_8),
((0),(0),\mathbb{Z}_8),(\mathbb{Z}_4,(0),(0)),((0),\mathbb{Z}_4,(0))\}$.

Let $J,K\in [I]$. Then $N[J]=N[K]$ and so by Lemma \ref{dimhelpnonred}, $J\in W$ or  $K\in W$, where $W$ is a strong resolving set of  $\Gamma(R)$. Thus one may assume that $V(\Gamma(R))\setminus A\subseteq W$.
Let

$V_{1}=\{((0),\mathbb{Z}_4,(0)), (\mathbb{Z}_4,(0), (0)),((0),(0),\mathbb{Z}_8)\}$ and

$V_{2}=\{(\mathbb{Z}_4,(0),\mathbb{Z}_8),((0),\mathbb{Z}_4,\mathbb{Z}_8), (\mathbb{Z}_4,\mathbb{Z}_4,(0))\}$.

Since  $d(I,J)=daim(\Gamma(R))$, for every  $I,J\in V_1$, we conclude that $I\in W$ or  $J\in W$. Thus one may let
$V(\Gamma(R))\setminus V_3\subseteq W$, where
$V_3=V_2\cup \{((0),(0),\mathbb{Z}_8)\}$.

 Since for every  $I,J\in V_3$ we have $N[I]\neq N[J]$, $I,J$   are strongly resolved by  some  vertices of $W=A(R)^*\setminus(V_3)$. This means that $sdim_M(\Gamma)=|W|=23-4=19$.
On the other hand, by Theorem \ref{isomorphismq} it is easily seen that   $sdim_M({\Gamma}(R))=23-6+2=19$.

We close this paper with the following result which completely  characterizes $sdim_M({\Gamma}(R))$, when $R$ is non-reduced.

\begin{cor}\label{isomorphismqwez}
 Let  $R\cong R_1\times\cdots \times R_n\times F_1\times \cdots \times F_m$, be a  ring,  $n\geq 1$, $m\geq1$ where $R_i$ is an Artinian local ring such that  for every $1\leq i\leq n$, $|A(R_i)^*|\geq 1$ and each $F_i$ is a field.  Then
   $sdim_M(\Gamma(R))=|V(\Gamma(R))|-2n-2m+2$.
\end{cor}
\begin{proof}
{The proof here is a refinement of the arguments in proofs of Theorem \ref{isomorphismq} and Lemma \ref{e32e}.
By similar proofs, one may get
 $\Gamma(R)_{SR}=H^{\prime \prime}+K_{|A(R_1)|}+K_{|A(R_2)|}+\cdots+K_{|A(R_n)|}$,
$\beta(H^{\prime \prime})=n+m-2$
and
 $\beta(K_{|A(R_1)|}+K_{|A(R_2)|}+\cdots+K_{|A(R_n)|})=n$.
Thus   $\beta({\Gamma}(R)_{SR})=2n+m-2$.
On the other hand, since
 $|V(\Gamma(R)_{SR})|=|V(\Gamma(R))|-m$,
we deduce that    $sdim_M(\Gamma(R))=|V(\Gamma(R))|-m-(2n+m-2)=|V(\Gamma(R))|-2n-2m+2$.
 }
\end{proof}

\noindent{\bf Acknowledgements.} Reza Nikandish in this work has been financially supported by the research of Jundi-Shapur Research Institute. The grant number was 01-100-1-1400.

{}

\end{document}